\documentclass[10pt]{article}

\usepackage{amsxtra,amssymb,amsthm,amsmath,latexsym}

\usepackage{graphicx}

\newtheorem{thm}{Theorem}

 \newcommand{\thmref}[1]{Theorem~\ref{#1}}

\newcommand{\R}{{\mathbb R}}

\newcommand{\dl}{{\delta}}
\newcommand{\bee}{\begin{equation*}}
\newcommand{\eee}{\end{equation*}}
\newcommand{\be}{\begin{equation}}
\newcommand{\ee}{\end{equation}}
\newcommand{\pn}{\par\noindent}

\title{Justification of the Dynamical Systems Method\\ 
(DSM) for global homeomorphisms}
\author{A G Ramm\\
\small Department of Mathematics\\[-0.8ex]
\small Kansas State University, Manhattan, KS 66506-2602, USA\\[-0.8ex]
\small \texttt{ramm@math.ksu.edu}\\
}

\begin{document}

\date{}
\maketitle
\begin{abstract} The Dynamical Systems Method (DSM) is justified
for solving operator
equations $F(u)=f$, where $F$ is a nonlinear operator in a Hilbert space 
$H$. It is assumed that $F$ is a global homeomorphism of $H$ onto $H$,
that $F\in C^1_{loc}$, that is, it has a continuous with respect to
$u$ Fr\'echet derivative $F'(u)$, that the operator $[F'(u)]^{-1}$
exists for all $u\in H$ and is bounded, $||[F'(u)]^{-1}||\leq m(u)$,
where $m(u)>0$ is a constant, depending on $u$, and not necessarily 
uniformly bounded with respect to $u$. It is proved under
these assumptions that the continuous analog of the Newton's method
$\dot{u}=-[F'(u)]^{-1}(F(u)-f), \quad u(0)=u_0, \quad (*)$
converges strongly to the solution of the equation $F(u)=f$
for any $f\in H$ and any $u_0\in H$. The global (and even local) existence 
of the solution to the Cauchy problem (*) was not established 
earlier without assuming that $F'(u)$ is Lipschitz-continuous.
The case when $F$ is not a global homeomorphism but a monotone operator
in $H$ is also considered.   
\end{abstract}

\pn{
 }

\noindent\textbf{Biographical note:} Prof. Alexander G. Ramm is an
author of more than 590 papers, 2 patents, 12 monographs, an editor
of 3 books. He is an associte editor of several Journals. He gave
more than 140 addresses at various Conferences and lectured at many
Universities in Europe, Asia, Australia and America. He won
Khwarizmi International Award in Mathematics, was a Mercator Professor 
in Germany,  an invited professor supported by the Royal Academy of 
Engineering in UK, was a London
Mathematical Society speaker, distinguished HKSTAM speaker, CNRS
research professor, Fulbright professor in Israel, distinguished
Foreign Professor in Mexico and Egypt. His research interests
include many areas of analysis, numerical analysis and mathematical
physics.\\

\section{Introduction}
Consider an operator equation: 
\be\label{e1} F(u)=f, \ee 
where $F$
is a nonlinear operator in a Hilbert space $H$.

We assume in this Section  that $F$ is a {\it global 
homeomorphism.}

For instance, $F$ may be a hemicontinuous  monotone 
operator such that a coercivity condition is satisfied, for example,
the following condition:
\be\label{e1'} \lim_{||u||\to \infty}\frac{(F(u),u)}{||u||}=\infty, \ee
where $(\cdot, \cdot)$ denotes the inner product in $H$ (see\cite{D}).
We assume that
$F\in C^1_{loc},$ i.e., the Fr\'{e}chet
derivative of $F$, $F'(u)$,  exists for every $u$ and depends continuously 
on $u$. Furthermore, we assume that  $[F'(u)]^{-1}$ exists and is 
bounded
for all $u\in H$, 
\be\label{e1"}
||[F'(u)]^{-1}||\leq m(u),
\ee 
where $m(u)$  depends on $u$  and {\it is not necessarily
uniformly bounded with respect to $u$. }

This assumption implies that $F$ is a {\it local} homeomorphism, but
it does not imply, in general, that $F$ is a {\it global} homeomorphism.
If $m(u)<m$, where $m>0$ is a 
constant independent of $u$, then it was proved in \cite{R489} that $F$
is a global homeomorphism.

While our main result in Section 1, Theorem 1, does not require
the monotonicity of $F$, the result in Section 2, Theorem 2,
will use the monotonicity of $F$.

We  assume in Section 2 that $F$
is monotone: 
\be\label{e2} F'(u)\geq 0\qquad \forall u\in H. \ee 
This means that $(F'(u)v,v)\geq 0$ for all $v\in H$.

In Remark 2, at the end of the paper, the following condition is 
mentioned:
\be\label{e3} \|F(u)\|<c  \Rightarrow\|u\|<c_1,\qquad c,c_1=const>0,
\ee 
which means that the preimages of bounded sets under the map $F$ are 
bounded sets. This condition does not hold for the operator
$F(u):=e^u$, $u\in \R$, $H=\R$, and that is why this monotone operator
$F$ is not surjective: equation $e^u=0$ does not have a solution in $H$.
    
By $c>0$ we denote in this paper various constants.

Our first main result, Theorem 1,  says that if $F\in C^1_{loc}$ is a 
global homeomorphism
and condition \eqref{e1"} holds, then a continuous analog
of the Newton's method (see equation \eqref{e4} below) converges globally,
that is, it converges for any initial approximation $u_0\in H$ and any 
right-hand side 
$f\in H$. 

{\it One of the  novel features of our result is the
absence of any smoothness assumptions on $F'(u)$: only the continuity of
$F'(u)$ with respect to $u$ is assumed.} 

In the earlier work 
(see \cite{R499}, \cite{R489}-\cite{R452}, \cite{R577}, and references 
therein,
except for \cite{R568} and \cite{R454}, \cite{R593}) it was 
often assumed that 
$F'(u)$ is Lipschitz continuous, or, at least, H\"older-continuous.

Our approach can be generalized to the case when $F$ is a local 
homeomorphism, if one uses the results in \cite{R575}.

In this paper for the first time no assumptions
on the smoothness of $F'(u)$ are made, only the continuity
of $F'(u)$ is assumed in a proof of the global existence of the solution
to the Cauchy problem (6), see Theorem 1 below. The author does not know 
of any way to prove even
the local existence of the solution to (6) without using the novel idea
and new method of the proof, given in the proof of Theorem 1.
The known methods do not seem to give any results even on the
local existence of the solution to problem (6) if $F'(u)$ is assumed
only continuous. Recall that the known Peano theorem fails
in infinite-dimensional Banach spaces. 
The standard assumption, that guarantees the local existence  
of the unique solution to the Cauchy problem (6) in an 
infinite-dimensional Banach space, is the 
Lipschitz condition for the operator $[F'(u)]^{-1}(F(u)-f)$,
which holds, in general, only if  $F'(u)$ is Lipschitz-continuous. 

In our second result, in Theorem 2 in Section 2, the operator $F$ is not 
assumed to be a global homeomorphism, and it is not assumed invertible
(injective), but it is assumed to be a monotone operator,
and it is assumed that equation  \eqref{e1} has a solution, possibly 
non-unique.

We give a Dynamical Systems Method (DSM) version for constructing the 
(unique) minimal-norm solution to
equation  \eqref{e1} with monotone operator $F$. This DSM version is a 
regularized continuous analog
of the Newton's method. We make no smoothness assumptions about $F'(u)$,
and assume only the continuity of $F'(u)$ with respect 
to $u$. 

Since we do not assume in Section 2 that
the operator $F'(u)$ is invertible in any sense, the problem, 
studied in this Section can be considered an ill-posed one.

Our proof of Theorem 2 contains new ideas and uses the 
ideas from the proof of Theorem 1.
 
Let us formulate our first result:
\begin{thm}\label{thm1}
If $F\in C^1_{loc}$ is a global homeomorphism and condition \eqref{e1"}  
holds, then  the 
problem \be\label{e4}
\dot{u}=-[F'(u)]^{-1}(F(u)-f),\quad u(0)=u_0; \
\dot{u}=\frac{du}{dt}, \ee 
is globally solvable for any $f$ and $u_0$ in
$H$, there exists 
the limit $u(\infty)=\lim_{t\to \infty} u(t),$ and
$F(u(\infty))=f$. 
\end{thm}
\begin{proof}
Denote 
\be\label{e4'} v:=F(u(t))-f.\ee Then 
$$\dot{v}=F'(u(t))\dot{u}=-v.$$ Thus,
problem \eqref{e4} is reduced to the following problem:
\be\label{e5} \dot{v}=-v,\quad v(0)=F(u_0)-f. \ee Problem \eqref{e5}
obviously has a unique global solution: 
\be\label{e6} v(t)=(F(u_0)-f)e^{-t}, \qquad \lim_{t\to \infty} 
v(t):=v(\infty)=0.
\ee 
Therefore, problem \eqref{e4} has a unique global solution. 

Let us explain the above statement in detail.
Consider an interval $[0,T]$, where $T>0$ is arbitrarily large. The
equation 
\be\label{e7} F(u(t))-f=v(t)\qquad 0\leq t\leq T, \ee
is uniquely solvable for $u(t)$ for any $v(t)$ because $F$ is a global
homeomorphism. 
Assumption \eqref{e1"},  the continuity of $F'(u)$ 
with respect to $u$, and the abstract inverse function theorem, imply
that the solution  $u(t)$ to equation \eqref{e7} is 
continuously differentiable with respect to $t$, because $v$ 
is continuously differentiable with respect to $t$ and $F$ is
continuously Fr\'echet differentiable with respect to $u$.
 
Differentiating \eqref{e7} and using relations \eqref{e5} and \eqref{e4'},
one gets the following equation:
\be\label{e8} F'(u(t))\dot{u}=\dot{v}=-v=-(F(u(t))-f). \ee
Using assumption \eqref{e1"}, one concludes from \eqref{e8} that $u=u(t)$ 
solves \eqref{e4} in the interval $t\in [0,T]$. Since $T>0$ is arbitrary,
$u=u(t)$ is a global solution to \eqref{e4}.

Since $\lim_{t\to \infty}v(t):=v(\infty)$ exists, and $F$ is a global 
homeomorphism, one concludes that the limit $\lim_{t\to 
\infty}u(t):=u(\infty)$ does exist.

Since $v(\infty)=0$, it follows that $F(u(\infty))=f$.

Theorem 1 is proved. 
\end{proof}
\textbf{Remark 1.} Theorem 1 implies that {\it any}  equation  \eqref{e1}
with $F$ being a global homeomorphism and $F\in C^1_{loc}$, 
such that condition \eqref{e1"} 
holds, can be solved by the DSM method \eqref{e4}, which is 
a continuous analog of the Newton's method.

\section{Finding the minimal-norm solution}

{\it Assumptions: In this Section we assume that $F\in C^1_{loc}$ is 
monotone,
that is, $F'(u)\geq 0$,
and assumptions  \eqref{e1"}- \eqref{e3} hold, but $F$ is not
a global homeomorphism, so that equation \eqref{e1} may have many 
solutions. We assume that \eqref{e1} has a solution. }

Since $F$ is monotone and continuous, and the set of 
of solutions to \eqref{e1} is non-empty, this set is closed and convex, so 
it has a {\it unique 
element with minimal norm} (see \cite{R499}). This element is called the 
minimal-norm 
solution to \eqref{e1}, and is denoted by $y$. 

{\it Our aim is to give a method 
for finding this element by a version of the DSM.}

Consider the problem \be\label{e13}
\dot{u}=-[F'(u)+a(t)I]^{-1}[F(u)+a(t)u-f],\quad u(0)=u_0, \ee where
$a\in C^1([0,\infty)),$ $\dot{a}<0$, 
\be\label{e14}
a(t)>0 \,\, \forall t\geq 0, \quad \lim_{t\to 
\infty}\frac{\dot{a}}{a}=0,\quad \lim_{t\to
\infty}a(t)=0. \ee

The assumptions of \thmref{thm1} do not hold 
for the operator  $F(\cdot)+a(t)I$ in the sense that
the quantity $\frac 1 {a(t)}$ in the estimate (14), see below, tends 
to infinity as $t\to \infty$. Let us explain this statement.

Under our {\it Assumptions}, the operator $F(\cdot)+a(t)I$ for every $t>0$
is a global homeomorphism
because $F$ is a monotone continuous operator and $a(t)>0$. 
One has
\be\label{e15'}\|[F'(u)+a(t)]^{-1}\|\leq \frac{1}{a(t)}.\ee
Therefore, the constant $m(u)$  for the operator $F(u)+a(t)u$
is $\frac{1}{a(t)}$. As
$t\to \infty$, this constant tends to infinity because $\lim_{t\to 
\infty}a(t)=0$.

Let us state our result:
\begin{thm}\label{thm2} Assume that $F\in C^1_{loc}$ is a monotone 
operator, equation
\eqref{e1} has a solution for the given $f$, and 
conditions  \eqref{e14} hold. Then
problem \eqref{e13} has a unique global
solution $u(t)$, there exists $u(\infty)$, and $u(\infty)=y$, where
$y$ is the minimal-norm solution to \eqref{e1}.
\end{thm}
\begin{proof}
Let 
\be\label{e14'}v(t)=F(u(t))+a(t)u(t)-f.\ee 
Then
\be\label{e15}\dot{v}=-v+\dot{a}(t)u(t),\qquad
v(0)=F(u_0)+a(0)u_0-f.\ee 
The map $u=G(v)$, where 
$$ v(t)=G^{-1}(u):=F(u)+a(t)u-f,$$ 
is a
local diffeomorphism for any $t\geq 0$, because $a(t)>0$ $\forall
t\geq 0$ and \eqref{e15'} holds. 

As in the
proof of \thmref{thm1} one concludes that the solution to \eqref{e13}
exists locally because the solution $v=v(t)$ to \eqref{e15} exists 
locally.

The solution to \eqref{e15} exists locally by the standard result, because  
the map $u=G(v)$ is $C^1_{loc}$ local diffeomorphism. The solution to 
\eqref{e15}
exists globally (see, e.g., \cite{R499}, p. 248) if
\be\label{ec} \sup_{t\geq 0}\|v(t)\|<c,\ee 
where $c>0$ here and below denote various estimation constants. 

{\it Let us briefly recall the proof
of this statement.} 

Assume that inequality \eqref{ec} holds, 
but the maximal interval 
of the existence of $v$ is finite, say, $[0,T)$, $T<\infty$.
 The length $\ell$ of the interval of the local existence of the
solution to the Cauchy problem \eqref{e15} depends only on the 
Lipschitz constant of $G(v)$ and on
the norm of the right hand side of \eqref{e15}. Both these quantities 
depend only on the constant $c$.
One solves the Cauchy problem
for equation \eqref{e15} with the initial data $v(T-0.5\ell)$
at the initial point $t=T-0.5\ell$. The unique solution to this problem
exists on the interval $[T-0.5\ell, T-0.5\ell+\ell)$.
 Consequently, $v$ exists on the interval $[0,T+0.5\ell)$ greater than 
$[0,T)$. This is a contradiction which proves that $T=\infty$.

The map $u=G(v)$ is $C^1_{loc}$ because it is inverse to the $C^1_{loc}$ 
map $v=F(u)+a(t)u-f:=G^{-1}(u):=Q(u),$ and $\|[Q'(u)]^{-1}\|\leq
\frac{1}{a(t)}<\infty$ for every $t\geq 0$. 

Therefore, the estimate $ \sup_{t\geq 0}\|v(t)\|<c$ holds if and only if 
\be\label{e15"}
\sup_{t\geq 0}\|u(t)\|<c,\ee
where $c>0$ stands for various constants.

Thus, to prove that $u(t)$
exists globally it is sufficient to prove inequality \eqref{e15"}. 

We prove this inequality, the existence of $u(\infty)$, and the
relation $u(\infty)=y$, by establishing two facts: 

a) the following inequality: 
\be\label{e16}
\|u(t)-w(t)\|\leq \frac{\|v(t)\|}{a(t)}, \ee 
and 

b) the limiting relation:
\be\label{e17}
\lim_{t\to \infty}\frac{\|v(t)\|}{a(t)}=0. \ee 
In formula \eqref{e16} $w(t)$ solves
the problem 
\be\label{e18} F(w)+a(t)w-f=0, \ee 
and $a(t)$ satisfies \eqref{e14}.
It is proved in
\cite{R499} that if $F$ is a monotone hemicontinuous operator and
equation \eqref{e1} has a solution, then equation \eqref{e18}
has a unique solution for any $f$ if $a(t)>0$, the limit  $w(\infty)$ 
exists, and $w(\infty)=y$. This, \eqref{e16}, and \eqref{e17} imply the 
existence of $u(\infty)$  and the relation $u(\infty)=y$.

Let us prove inequality \eqref{e16}. Since $F$ is monotone, one
has 
$$(F(u)-F(w),u-w)\geq 0,$$ 
so 
\be\label{e19} (v,u-w)=(F(u)-F(w)+a(t)(u-w),u-w)\geq
a(t)\|u-w\|^2. \ee Applying the Cauchy inequality to the left side
of \eqref{e19}, one gets \eqref{e16}. 

Let us prove \eqref{e17}.
Denote 
\be\label{eh}h(t):=\|v(t)\|.\ee 
Multiply equation  \eqref{e15} by $v$ and get
\be\label{e20'} h\dot{h}\leq -h^2+|\dot{a}|\|u(t)\|h.\ee 
If $h(t)> 0$, one obtains from  \eqref{e20'} the following inequality
\be\label{e20} \dot{h}(t)\leq
-h(t)+|\dot{a}(t)|(\|u(t)-w(t)\|+\|w(t)\|). \ee 
If $h(t)=0$ on some interval $t\in (a,b)$, then 
$\dot{h}=0$ on this interval, and the above inequality holds trivially.
If $h(t)=0$ at an isolated point $t=s$, i.e.,  $h(s)=0$, then \eqref{e20} 
holds by continuity at $s+0$. The existence of the derivative 
$\dot{h}(s+0)$ at the point $s$ at which $h(s)=0$ can be checked using
the definition of the one-sided derivative: 
\be\label{e20a}\dot{h}(s+0)=\lim_{\tau\to 
+0}[h(s+\tau)-h(s)]/\tau=\lim_{\tau\to 
+0}h(s+\tau)/\tau.
\ee
Since $v(t)$ is continuously differentiable, one has 
$h(s+\tau):=||v(s+\tau)||=||\tau\dot{v}(s)+o(\tau)||$.
Therefore the limit in \eqref{e20a} exists and is equal to 
$||\dot{v}(s)||$. This limit is denoted $\dot{h}(s)$.
Thus, inequality \eqref{e20} holds for all $t\ge 0$.

Since $w(\infty)$ exists, one has 
\be\label{ew}\sup_{t\geq 0}\|w(t)\|<c.\ee 
Using \eqref{ew} and  \eqref{e16}, one
gets from \eqref{e20} the inequality 
\be\label{e21} \dot{h}\leq
-h+\frac{|\dot{a}(t)|}{a(t)}h(t)+|\dot{a}(t)|c. \ee 
Let us derive from inequality  
\eqref{e21} the desired conclusion \eqref{e17}.

Fix an arbitrary
small $\dl>0$. The first assumption \eqref{e14} implies that
\be\label{e22} \frac{|\dot{a}(t)|}{a(t)}\leq \dl \quad for \ t\geq
t_\delta. \ee 
Using the well-known Gronwall inequality, one obtains from  \eqref{e21} 
the following inequality
\be\label{e23} h(t)\leq
h(t_\dl)e^{-(1-\dl)(t-t_\delta)}+c\int_{t_\dl}^te^{-(1-\dl)(t-s)}|\dot{a}(s)|ds.
\ee 
Let us divide both sides of \eqref{e23} by $a(t)$ and prove that
the following two relations hold:
\be\label{e24} \lim_{t\to \infty}\frac{e^{-(1-\dl)t}}{a(t)}=0, \ee
and \be\label{e25} \lim_{t\to \infty}\frac{\int_{t_\dl}^t
e^{(1-\dl)s}|\dot{a}(s)|ds}{e^{(1-\dl)t}a(t)}=0. \ee This will
complete the proof of \thmref{thm2}. 

From inequality  \eqref{e22} one gets
\be\label{e26} ce^{-\dl t}\leq a(t). \ee
 This implies  relation \eqref{e24} if $\dl<\frac{1}{2}.$ 

Applying the L'Hospital rule one
 proves relation  \eqref{e25} because 
$$\lim_{t\to
 \infty}\frac{|\dot{a}(t)|}{(1-\dl)a(t)+\dot{a}(t)}=0,$$ 
as follows from the second
 assumption \eqref{e14} provided that $\delta<1$. \\
 \thmref{thm2} is proved. \hfill $\Box$

\textbf{Remark 2.} The equation $e^u=0$, $u\in \R$, $H=\R$, does
not have a solution, although $F(u)=e^u$ is monotone, $F'(u)=e^u>0$
is boundedly invertible for every $u\in \R$ and
$\|[e^u]^{-1}\|=e^{-u}\leq m(u)<\infty$ for every $u\in \R$.
Assumption \eqref{e3} is not satisfied in this example, and this
is the reason for the unsolvability of the equation $e^x=0.$ Note that
$e^{x}\leq c$ as $x\to -\infty$, so assumption \eqref{e3} does not hold.

In recent papers \cite{R570} and \cite{R587} some nonlinear differential 
inequalities are derived and used for a study of the large time behavior 
of solutions to evolution problems.

\end{proof}


\begin{thebibliography}{00}

\bibitem{D} K. Deimling, {\bf Nonlinear functional analysis}, 
Springer-Verlag, Berlin, 1985.

\bibitem{R568} N. S. Hoang, A. G. Ramm,  Existence of solution to an 
evolution equation
and a justification of the DSM for equations with monotone
operators, Comm. Math.Sci., 7, N4, (2009), 1073-1079.

\bibitem{R577} N. S. Hoang, A. G. Ramm,
 The Dynamical Systems Method for solving
nonlinear equations with monotone operators,
Asian Europ. Math. Journ., 3, N1, (2010), 57-105.

\bibitem{R593} N. S. Hoang, A. G. Ramm,
DSM of Newton-type for solving operator equations $F(u)=f$
with minimal smoothness assumptions on $F$,
International Journ. Comp.Sci. and Math. (IJCSM),
3, N1/2, (2010), 3-55.



\bibitem{R499} A. G. Ramm, {\bf Dynamical systems method for solving
operator equations}, Elsevier, Amsterdam, 2007.


\bibitem{R489}
 A. G. Ramm,  Dynamical systems method and a homeomorphism
theorem, Amer. Math. Monthly, 113, N10, (2006), 928-933.

\bibitem{R485} A. G. Ramm, Dynamical systems method (DSM) and
nonlinear problems, in the book: Spectral Theory and Nonlinear 
Analysis,
World Scientific Publishers, Singapore, 2005, 201-228. (ed J.
Lopez-Gomez).

\bibitem{R469} A. G. Ramm, DSM for ill-posed equations with monotone 
operators,
Comm. in Nonlinear Sci. and Numer. Simulation, 10, N8, (2005),935-940.

\bibitem{R456} A. G. Ramm, Dynamical systems method for solving
operator equations, Communic. in Nonlinear Sci. and
Numer. Simulation, 9, N2, (2004), 383-402.

\bibitem{R454} A. G. Ramm,
Dynamical systems method for solving nonlinear
operator equations, International Jour. of
Applied Math. Sci., 1, N1, (2004), 97-110.

\bibitem{R452} A. G. Ramm, Global convergence for ill-posed equations
with monotone operators: the dynamical systems method, J.
Phys A, 36, (2003), L249-L254.

\bibitem{R575} A. G. Ramm,
Implicit Function Theorem via the DSM,
Nonlinear Analysis: Theory, Methods and Appl.,
72, N3-4, (2010), 1916-1921.

\bibitem{R570} A. G. Ramm, Asymptotic stability of solutions to abstract 
differential
equations, Journ. of Abstract Diff. Equations and Applications (JADEA),
1, N1, (2010), 27-34.

\bibitem{R587} A. G. Ramm, A nonlinear inequality and evolution problems,
Journ, Ineq. and Special Funct., (JIASF), 1, N1, (2010), 1-9.

\end{thebibliography}
\end{document}